\documentclass[11pt]{article}
\usepackage[latin1]{inputenc}
\usepackage{amsmath,amsthm,amssymb}
\usepackage{amsfonts}
\usepackage{amsmath,amsthm,amssymb,amscd}
\usepackage{latexsym}
\usepackage{color}
\usepackage{graphicx}
\usepackage{mathrsfs}

\makeatletter \@addtoreset{equation}{section} \makeatother

\newtheorem{theorem}{Theorem}[section]

\newtheorem{proposition}{Proposition}[section]
\newtheorem{lemma}{Lemma}[section]
\newtheorem{remark}{Remark}[section]

\newcommand{\B}{\Bigr}

\begin{document}

\title{\sc On the Schr\"{o}dinger-Poisson system \\ with $(p,q)$--Laplacian }
\author{\sc {Yueqiang Song$^{a}$, Yuanyuan Huo$^{a}$, and Du\v{s}an D.
Repov\v{s}$^{b,c,d}$\thanks{{ { Corresponding author } }}
\thanks{{{E-mail:
songyq16@mails.jlu.edu.cn (Song), 
hyy13353223575@163.com (Huo), 
dusan.repovs@guest.arnes.si (Repov\v{s})} }}}\\[1mm]
$^{\small\mbox{a}}${\small College of Mathematics, Changchun Normal
University,   Changchun, 130032,  P.R. China}\\[-2mm]
$^{\small\mbox{b}}$ {\small Faculty of Education, University of Ljubljana, Ljubljana, 1000,
Slovenia}\\[-2mm]
$^{\small\mbox{c}}$ {\small Faculty of
Mathematics and Physics, University of Ljubljana, Ljubljana, 1000,
Slovenia}\\[-2mm]
$^{\small\mbox{d}}$ {\small Institute of Mathematics, Physics and Mechanics, Ljubljana, 1000,
Slovenia}
}
\date{}
\maketitle

\begin{abstract}
We  study a class of  Schr\"{o}dinger-Poisson systems
with $(p,q)$--Laplacian. Using fixed point theory, we obtain a  new
existence result for nontrivial solutions. The main novelty of the
paper is the combination of a double phase operator and the
nonlocal term. Our results generalize some known results.
\end{abstract}


\emph{\it Math. Subj. Classif. (2020):} 35J47; 35J60; 35R11.

\section{ Introduction }
In this article, we shall study the following
Schr\"{o}dinger-Poisson system with $(p,q)$--Laplacian
\begin{eqnarray}\label{e1.1}
\left\{\begin{array}{l}
-\Delta_{p}u-\Delta_{q}u+(|u|^{p-2}+|u|^{q-2})u -\phi|u|^{q-2}u=
h(x,u) + \lambda g(x)    \quad \mbox{in }\,\,\mathbb{R}^3,\\
-\Delta\phi = |u|^q \quad \mbox{in }\,\,\mathbb{R}^3,
\end{array}\right.
\end{eqnarray}
where $\Delta_{\varsigma}=$div$(|\nabla u|^{p-2}\nabla u)$  is the
$\varsigma$-Laplacian ($\varsigma = p, q$), $\frac{3}{4} < p < q <
3$,  $\lambda$ is a positive parameter, the nonnegative function $g
\in L^{\frac{3q}{4q-3}}(\mathbb{R}^3)$ is a perturbation term, and
$g(x) \not\equiv 0$. Here, $h: \mathbb{R}^3 \times \mathbb{R}\to
\mathbb{R}$ is a Carath\'{e}odory function and it satisfies certain
assumptions.

Our study of problem \eqref{e1.1} was motivated by two main reasons.
On the one hand, when $p=q=2$ and $\lambda \equiv 0$, problem
\eqref{e1.1} becomes the following nonlinear Schr\"{o}dinger-Poisson
system
\begin{eqnarray}\label{e1.2}
\left\{\begin{array}{l}
-\Delta u+ u -\phi u = h(x,u) \quad \mbox{in }\,\,\mathbb{R}^3,    \\
-\Delta\phi = |u|^2 \quad \mbox{in }\,\,\mathbb{R}^3.
\end{array}\right.
\end{eqnarray}
System \eqref{e1.2} depicts how charged particles interact
with the motion  electromagnetic field. While the nonlocal term
$\phi u$ describes interactions with the electric field, the
nonlinear term models interactions between the particles. By virtue
of its strong physical background, system \eqref{e1.2} has
drawn wide attention in recent decades. For $p=q\neq2$, system
\eqref{e1.1} was studied for the first time  by Du et al. \cite{du1} 
and the existence of nontrivial solutions of the system was obtained
by invoking the Mountain Pass Theorem. For the quasilinear
Schr\"{o}dinger-Poisson system, we refer to Du et al. \cite{du2}.
Readers interested in learning more about the results for the
Schr\"{o}dinger-Poisson system using the variational methods are
referred to Ambrosetti-Ruiz \cite{am}, D'Aprile-Mugnai \cite{da},
Ruiz \cite{ru} and the references therein.

On the other hand,  when $p\neq q$, problem \eqref{e1.1} is driven
by a differential operator with unbalanced growth. When problem
\eqref{e1.1} without the nonlocal term $\phi u$  becomes a $p\&
q$-Laplacian equation
\begin{eqnarray}\label{e1.3}
-\Delta_{p}u-\Delta_{q}u+(|u|^{p-2}+|u|^{q-2})u -\phi|u|^{q-2}u=
h(x,u)  \quad \mbox{in }\,\,\mathbb{R}^N,
\end{eqnarray}
this problem has a rich physical background, since the double phase
operator has been applied to describe steady-state solutions of
reaction diffusion problems in biophysics, plasma physics, and
chemical reaction analysis. Using the variational methods, some
results for problem \eqref{e1.3} can be found in Bartolo et al.
\cite{ba}, Figueiredo \cite{fi}, Papageorgiou et al. \cite{pa},
 and
the references therein.

Inspired by the above references,  we prove in this paper the
existence  of nontrivial solutions for problem \eqref{e1.1} by using
 fixed point theory. Although some authors have
already used fixed point
theory, see  Carl-Heikkil\"{a} \cite{ca},  de Souza \cite{so} and
Tao-Zhang \cite{t1, t2},  as far as we know, problem \eqref{e1.1}
has not been studied before. Because of the occurence of a
nonhomogeneous term, we can  prove that a weak solution to problem
\eqref{e1.1} exists by the fixed point theory. The results in this
paper can be regarded as an extension of results in Du et al.
\cite{du1, du2} and Tao-Zhang \cite{t1, t2}.
In some sense, our
results are new, even in the $p = q$ case.

Our existence result, which is the main result of this paper, can be
stated as follows.
\begin{theorem}\label{T1.1}
Assume that 
$h(x,u): \mathbb{R}^3 \times \mathbb{R}^+\to
\mathbb{R}^+:=[0, +\infty)$ is a nondecreasing function in $u$, and
$h(x,u) =0$ when $u < 0$. Moreover,
 assume that it satisfies the following
condition
\begin{equation}\label{e1.4}
|h(x,t)|\leq d_1(x)|t|^{\tau-1}+d_2(x)|t|^{q^\ast-1},  \ \text{for\
all}\ (x, t) \in \mathbb{R}^3 \times \mathbb{R},\end{equation} where
$q \leq \tau < q^\ast := \frac{3p}{3-p}$, $0 \leq d_1 \in
L^\eta(\mathbb{R}^3)$, $0 \leq d_2 \in L^\infty(\mathbb{R}^3)$, and
$\eta = \frac{6}{6-\tau}$. Then there exists $\lambda_{0}>0$ such
that for every $0<\lambda\leqslant\lambda_{0}$, problem \eqref{e1.1}
has a 
positive solution.
\end{theorem}
\begin{remark}\label{rmk1.1} We point out that
there are many functions that satisfy the assumptions of Theorem
\ref{T1.1}. For example, we can take
$h(x,t)=\frac{1}{1+x^2}|t|^{q^\ast-1}$.\end{remark}

\section{Preliminaries}\label{sec preliminaries}
In this section, we shall present  some preliminary results, as well
as some notations and useful results. To this end, let $W$ be the
subspace of $W^{1,p}(\mathbb{R}^3)$ and $W^{1,q}(\mathbb{R}^3)$,
defined by $ W = W^{1,p}(\mathbb{R}^3) \bigcap
W^{1,q}(\mathbb{R}^3)$,  with respect to the norm $\Vert u \Vert =
\Vert u \Vert_{W^{1,p}(\mathbb{R}^3)} + \Vert u
\Vert_{W^{1,q}(\mathbb{R}^3)}$. Since $W^{1,r}(\mathbb{R}^3)$, with
$1 < r < \infty$, is a separable reflexive Banach space, we deduce
that $W$ is a separable reflexive Banach space. Moreover,  we also
know that the embeddings $W\hookrightarrow L^{p}(\mathbb{R}^3),
L^{q}(\mathbb{R}^3)$ are continuous.
On the other hand, according to Du et al. \cite{du2}, for any given
$u\in W^{1,q}(\mathbb{R}^{3})$, there exists a unique
\[\phi_{u}(x)=\frac{1}{4\pi}\int_{\mathbb{R}^{3}}\frac{|u(y)|^{q}}{|x-y|}\mathrm{d}y,\quad \phi_{u}\in D^{1,2}(\mathbb{R}^{3}),\]
satisfying $ -\Delta\phi_{u}=|u|^{q}$.  

We now summarize the
properties of $\phi_{u}$  which will be used later.
\newtheorem{lem}{Lemma}[section]
\begin{lem} (Du et al. \cite{du2}\label{lemma2.1})
    Let $u\in W^{1,q}(\mathbb{R}^{3})$. Then the following properties hold:\\
    $(1)\ \phi_{u}\geqslant 0,$ for all $x\in\mathbb{R}^{3};$\\
    $(2)$ For any $t\in \mathbb{R}^{+},\phi_{tu}=t^{q}\phi_{u}$, and $\phi_{u_{t}}t^{kq-2}\phi_{u}(tx)$ with $u_{t}(x)=t^{k}u(tx);$\\
    $(3)\;\|\phi_{u}\|_{D^{1,2}}\leqslant C\|u\|^{q}$, where $C$ is independent of $u;$\\
    $(4)$\;If $u_{n}\rightharpoonup u$ in $W^{1,q}(\mathbb{R}^{3})$, then $\phi_{u_{n}}\rightharpoonup \phi_{u}$ in $D^{1,2}(\mathbb{R}^{3})$, and
$\int_{\mathbb{R}^{3}}\phi_{u_{n}}|u_{n}|^{q-2}u_{n}\varphi
\mathrm{d}x\to \int_{\mathbb{R}^{3}}\phi_{u}|u|^{q-2}u\varphi
\mathrm{d}x$, for all $\varphi\in W^{1,q}(\mathbb{R}^{3})$.
\end{lem}
Substituting $\phi=\phi_{u}$ into system \eqref{e1.1}, we can
rewrite  \eqref{e1.1} as a single equation
\begin{equation}
-\Delta_{p}u  -\Delta_{q}u
+(|u|^{p-2}u+|u|^{q-2}u)-\phi_{u}|u|^{q-2}u=h(x,u)+\lambda g(x),\  \hbox{for all} \
u\in W.\label{2.1}\tag{2.1}
\end{equation}
We define the energy functional $I$ on $W$ by
\begin{align*}
I(u)=\frac{1}{p}\int_{\mathbb{R}^{3}}\bigl(|\nabla
u|^{p}+|u|^{p}\bigr)\mathrm{d}x
+\frac{1}{q}\int_{\mathbb{R}^{3}}\bigl(|\nabla
u|^{q}+|u|^{q}\bigr)\mathrm{d}x -
\frac{1}{2q}\int_{\mathbb{R}^{3}}\phi_{u}|u|^{q}\mathrm{d}x-\int_{\mathbb{R}^3}\B(
H(x,u)+\lambda \textsl{g}(x)\B)\mathrm{d}x,
\end{align*}
where $H(x,t)=\int_0^th(x,s)ds.$
 It is straightforward to show
that $I\in C^{1}(W, \mathbb{R})$ and
\begin{align*}
\langle I'(u),\psi\rangle =&\int_{\mathbb{R}^{3}}\bigl(|\nabla
u|^{p-2}\nabla u\nabla \psi+ |u|^{p-2}u\psi\bigr)\mathrm{d}x 
+
\int_{\mathbb{R}^{3}}\bigl(|\nabla u|^{q-2}\nabla u\nabla \psi+
|u|^{q-2}u\psi\bigr)\mathrm{d}x
\\
&-\int_{\mathbb{R}^{3}}\phi_{u}|u|^{q-2}u\psi \mathrm{d}x
-\int_{\mathbb{R}^3}\B( h(x,u)+\lambda
\textsl{g}(x)\B)\psi\mathrm{d}x.
\end{align*} It is
easy to verify that $(u,\phi_{u})\in W\times
D^{1,2}(\mathbb{R}^{3})$ is a solution of system \eqref{e1.1} if and
only if $u\in W$ is a critical point of $I$.

Now, we introduce the necessary fixed-point theorem due to
Carl-Heikkil\"{a} \cite{ca},  which plays a crucial role in proving
our conclusions. For this, let $\mathcal {E}$ be a real Banach
space. A nonempty subset $\mathcal {E}_{+}\neq\{0\}$ of $\mathcal
{E}$ is called an order cone if it satisfies the following conditions: ${(a)}$ $\mathcal
{E}_{+}$ is closed and convex; ${(b)}$ if $v\in \mathcal {E}_{+}$
and $\delta\geqslant0$, then $\delta v\in \mathcal {E}_{+}$; ${(c)}$
if $v\in \mathcal {E}_{+}$ and $-v\in \mathcal {E}_{+}$, then $v=0$.
An order cone $\mathcal {E}_{+}$ induces a partial order in $W$ in
the following way: $x\preceq y$ and only if $y-x\in \mathcal
{E}_{+}$, and $(W,\preceq)$ is called an ordered Banach space. If
$\inf\{x,y\}$ and $\sup\{x,y\}$ exist for all $x,y\in W$ with
respect to $\preceq$, then $(W,\Vert \cdot \Vert )$ is called a
lattice. In addition, if $\Vert x^{\pm} \Vert \leqslant \Vert x
\Vert $ for each $x\in W$, where $x^{+}:=\sup\{0,x\}$ and
$x^{-}:=-\inf\{0,x\}$, then $(W,\Vert \cdot \Vert )$ is a Banach
semilattice.  We also note that the dual space $W'$ of $W$ has the
following partial order:
\begin{align*}
\varphi_{1},\varphi_{2}\in W',\varphi_{1}\lhd
\varphi_{2}\Leftrightarrow \langle\varphi_{1},v\rangle
\leqslant\langle\varphi_{2},w\rangle, \ \text{for all} \ w\in
\mathcal {E}_{+}.
\end{align*}

Next, we give the definition of fixed point property, which,
according to Carl-Heikkil\"{a} \cite{ca}, is the following one: $P$
is said to have a fixed point property if each increasing mapping
$G:P\to P$ has a fixed point.

\begin{proposition}\label{pro2.1} (Carl-Heikkil\"{a} \cite[Corollary 2.2]{ca})
Let $W$ be a reflexive Banach semilattice. Then every closed ball in
$W$ has the fixed point property.
\end{proposition}

\section{Proof of Theorem \ref{T1.1}}
In order to prove Theorem \ref{T1.1}, we first prove some key
lemmas. To begin, we define the functional $\mathcal {B}:
W\rightarrow W'$ by $ \langle \mathcal {B}u,v\rangle=
\int_{\mathbb{R}^{3}}(|\nabla u|^{p-2}\nabla u\nabla v +
|u|^{p-2}uv)\mathrm{d}x + \int_{\mathbb{R}^{3}}(|\nabla
u|^{q-2}\nabla u\nabla v + |u|^{q-2}uv)\mathrm{d}x.$ Clearly,
$\mathcal {B}u$ is linear for all $u\in W$. This means that the
H\"{o}lder inequality holds
\begin{equation*}
|\langle \mathcal {B}u,v\rangle|\leqslant C_1\Vert u \Vert^{p-1}
\Vert v \Vert + C_2\Vert u \Vert^{q-1} \Vert v \Vert, \ \text{ for\
some }\ C_1, C_2 >0.
\end{equation*}
Therefore $\mathcal {B}u\in W'$ and $\mathcal {B}$ are
well-defined. In addition, we have the following property of
$\mathcal {B}$.
\begin{lemma}\label{lem3.1} The operator $\mathcal {B}: W\rightarrow W'$
is continuous and invertible.\end{lemma}
\begin{proof} Let $\{u_{k}\}$ in $W$ be such that $u_{k}\rightarrow u$ in $W$.
Using the  H\"{o}lder  inequality, for $v\in W$ with $\Vert v \Vert
\leqslant 1$, we have
\begin{equation*}
\left\|\mathcal {B}u_{k}-\mathcal {B}u\right\|_{W'}=\sup_{v\in
W,\left\|v\right\|\leqslant1}|\langle \mathcal {B}u_{k}-\mathcal
{B}u,v\rangle|\leqslant \|u_{k}-u\|_{W^{1,p}(\mathbb{R}^3)}^p
+ \|u_{k}-u\|_{W^{1,q}(\mathbb{R}^3)}^q \rightarrow 0.
\end{equation*}
This means that the operator $\mathcal {B}$ is continuous.
Considering $p, q\geqslant2$ and $\langle \mathcal {B}u,u\rangle=
\left\|u\right\|^{p} + \left\|u\right\|^{q}$ for all $u\in W$, we
have $ \lim_{\left\|u\right\|\rightarrow\infty}\frac{\langle
\mathcal {B}u,u\rangle}{\left\|u\right\|}=\infty$. It well-known
that
\begin{align*}
(|a |^{s-2}a-|b|^{s-2}b|)(a-b)\geqslant C_{p}|a-b|^{s-2},\
\text{for all}\ s \geq 2, a,b\in\mathbb{R},
\end{align*}
and we have $\langle \mathcal {B}u_{1}-\mathcal {B}u_{2},
u_{1}-u_{2}\rangle>0$, for all $u_{1}, u_{2}\in W,\ u_{1}\neq
u_{2}$.  Therefore, by the Minty-Browder Theorem (see
\cite[Theorem 5.16]{br}), we obtain that the operator $W$ is
reversible. Hence the proof of Lemma \ref{lem3.1} is complete.
\end{proof}

Similar to the proof of Lemma 3.2 in \cite{t1}, we can show that
$\mathcal {B}^{-1}:(W',\lhd )\rightarrow (W,\preceq)$ is increasing.

Next,  inspired by \cite{t3}, let the operator $\mathcal {T} :
W\rightarrow W'$ be defined by
\begin{align*}
\langle \mathcal {T}u,v\rangle
=\int_{\mathbb{R}^{3}}\B(\phi_{u^+}|u^+|^{q-2}u^+
+ h(x,u^+)+\lambda g(x)\B) v\mathrm{d}x, \ \text{for all}\  u,
v\in W,
\end{align*}
where $u^+:= \max\{u, 0\}$ and $u^-:= -\min\{u, 0\}.$
 By the
H\"{o}lder inequality, the Sobolev Embedding Theorem, and the
Hardy-Littlewood-Sobolev inequality, there exist some positive
constants $C^\ast$, $C^{\ast\ast}$ and $C^{\ast\ast\ast}$ such that
\begin{align}\label{e3.1}
\left|\langle \mathcal {T}u,v\rangle\right| \leq
\left(C^\ast\|
u^+
\|^{2q-1} +
C^{\ast\ast}\|d_1\|_\eta\|u^+\|^{\tau-1}  +
C^{\ast\ast\ast}\|d_2\|_\infty\|u^+\|^{q^\ast-1} +
\lambda\|g\|_{\frac{3q}{4q-3}}\right)\|v\|.
\end{align}

Let  $\mathcal {G} :=\mathcal {B}^{-1}\circ \mathcal {T}$. Then we
have the following result.
\begin{lemma}\label{lem3.5} Under the hypotheses of Theorem
1.1, for any $0<\lambda\leqslant\lambda_{0}$, there exists $R>0$,
such that $\mathcal {G}(\mathbb{B}_{W}[0,R])\subset
\mathbb{B}_{W}[0,R]$,  where $\mathbb{B}_{W}[0,R]=\{u\in W :
\|u\|\leqslant R\}$.
\end{lemma}
\begin{proof}
Let $u\in W$, $v=(\mathcal {B}^{-1}\circ \mathcal {T})u=\mathcal
{G}u$. We note that  $\langle \mathcal {B}v,
v\rangle=\|v\|_{W^{1,p}(\mathbb{R}^3)}^{p} +
\|v\|_{W^{1,q}(\mathbb{R}^3)}^{q}$. We consider 3 possible
cases:

\noindent {\it Case 1}. $\|v\|_{W^{1,q}(\mathbb{R}^3)}\geq 1.$ Then
$\|v\|_{W^{1,q}(\mathbb{R}^3)}^{q}\geq\|v\|_{W^{1,q}(\mathbb{R}^3)}^{p},$
hence
\begin{equation}\label{e3.2}
\langle \mathcal {B}v, v\rangle \geqslant
\|v\|_{W^{1,p}(\mathbb{R}^3)}^{p}+\|v\|_{W^{1,q}(\mathbb{R}^3)}^{p}
\geqslant
2^{1-p}(\|v\|_{W^{1,q}(\mathbb{R}^3)}+\|v\|_{W^{1,p}(\mathbb{R}^3)})^{p}
= 2^{1-p}\|v\|^{p}.
\end{equation}
\noindent {\it Case 2.} $\|v\|_{W^{1,q}(\mathbb{R}^3)}< 1$ and
$\|v\|_{W^{1,p}(\mathbb{R}^3)}\geq 1.$ Then
$\|v\|_{W^{1,p}(\mathbb{R}^3)}\geq 1 >
\|v\|_{W^{1,q}(\mathbb{R}^3)}$. Since
$\|v\|=\|v\|_{W^{1,q}(\mathbb{R}^3)}+\|v\|_{W^{1,p}(\mathbb{R}^3)}$,
we  get $2\|v\|_{W^{1,q}(\mathbb{R}^3)} \leq \|v\| \leq
2\|v\|_{W^{1,p}(\mathbb{R}^3)},$ therefore
\begin{align}\label{e3.3}
\langle \mathcal {B}v,
v\rangle=\|v\|_{W^{1,p}(\mathbb{R}^3)}^{p}+\|v\|_{W^{1,q}(\mathbb{R}^3)}^{q}\geqslant\|v\|_{W^{1,p}(\mathbb{R}^3)}^{p}
\geqslant\frac{1}{2^{p}}\|v\|^{p}.
\end{align}
\noindent {\it Case 3.} $\|v\|_{W^{1,q}(\mathbb{R}^3)}< 1$ and
$\|v\|_{W^{1,p}(\mathbb{R}^3)}< 1.$ Then
$\|v\|_{W^{1,p}(\mathbb{R}^3)}^{q}\leqslant\|v\|_{W^{1,p}(\mathbb{R}^3)}^{p}$,
therefore
\begin{align}\label{e3.4}
\langle \mathcal {B}v, v\rangle
\geqslant\|v\|_{W^{1,p}(\mathbb{R}^3)}^{q}+\|v\|_{W^{1,q}(\mathbb{R}^3)}^{q}
\geqslant2^{1-q}(\|v\|_{W^{1,q}(\mathbb{R}^3)}+\|v\|_{W^{1,p}(\mathbb{R}^3)})^{q}
 =2^{1-q}\|v\|^{q}.
\end{align}
From \eqref{e3.2}, \eqref{e3.3} and \eqref{e3.4}, we have
\begin{align}\label{e3.5}
\langle \mathcal {B}v, v\rangle  \geqslant 2^{-p}\|v\|^p
\quad\mbox{or}\quad \langle \mathcal {B}v, v\rangle  \geqslant
2^{1-q}\|v\|^q.
\end{align}
On the other hand, we have
\begin{equation}\label{e3.6}
\|\mathcal {G}u\|_{W^{1,p}(\mathbb{R}^3)}^{p} + \|\mathcal
{G}u\|_{W^{1,q}(\mathbb{R}^3)}^{q}  = \langle \mathcal {T}u,
\mathcal {G}u\rangle \leq \|\mathcal {T}u\|_{W'}\|\mathcal {G}u\|.
\end{equation}
If $\|u\|\leqslant R,$ then by   \eqref{e3.1}, \eqref{e3.5} and
\eqref{e3.6},  one has
\begin{align*}
2^{-p}\|\mathcal {G}u\|^{p-1} &\leqslant\|\mathcal
{T}u\|_{W'}\leqslant C^\ast\|u\|^{2q-1} +
C^{\ast\ast}\|d_1\|_\eta\|u\|^{\tau-1}  +
C^{\ast\ast\ast}\|d_2\|_\infty\|u\|^{q^\ast-1} +
\lambda\|g\|_{\frac{3q}{4q-3}}\\
&\leqslant C^\ast R^{2q-1} + C^{\ast\ast}\|d_1\|_\eta R^{\tau-1} +
C^{\ast\ast\ast}\|d_2\|_\infty R^{q^\ast-1} +
\lambda\|g\|_{\frac{3q}{4q-3}}
\end{align*}
and
\begin{align*}
2^{1-q}\|\mathcal {G}u\|^{q-1} \leqslant C^\ast R^{2q-1} +
C^{\ast\ast}\|d_1\|_\eta R^{\tau-1} + C^{\ast\ast\ast}\|d_2\|_\infty
R^{q^\ast-1} + \lambda\|g\|_{\frac{3q}{4q-3}}.
\end{align*}
From this, we obtain
\begin{align}\label{e3.7}
\frac{\|\mathcal {G}u\|^{p-1}}{R^{p-1}}\leqslant 2^{p}C^\ast
R^{2q-p}+2^{p}C^{\ast\ast}\|d_1\|_\eta
R^{\tau-p}+2^{p}C^{\ast\ast\ast}\|d_2\|_\infty\|d_2\|_{\infty}R^{q^\ast-p}
+2^{p}\lambda\frac{\|g\|_{\frac{3q}{4q-3}}}{R^{p-1}}.
\end{align}
Similarly, we can also get
\begin{align}\label{e3.8}
\frac{\|\mathcal {G}u\|^{q-1}}{R^{q-1}}\leqslant 2^{q-1}C^\ast
R^{q}+2^{q-1}C^{\ast\ast}\|d_1\|_\eta
R^{\tau-q}+2^{q-1}C^{\ast\ast\ast}\|d_2\|_{\infty}R^{q^\ast-q}
+2^{q-1}\lambda \frac{\|g\|_{\frac{3q}{4q-3}}}{R^{q-1}}.
\end{align}
We now take $R>0$ sufficiently small so that
\begin{align*}
2^{p}C^\ast R^{2q-p}+2^{p}C^{\ast\ast}\|d_1\|_\eta
R^{\tau-p}+2^{p}C^{\ast\ast\ast}\|d_2\|_\infty\|d_2\|_{\infty}R^{q^\ast-p}\leqslant\frac{1}{2}
\end{align*}
and
\begin{align*}
2^{q-1}C^\ast R^{q}+2^{q-1}C^{\ast\ast}\|d_1\|_\eta
R^{\tau-q}+2^{q-1}C^{\ast\ast\ast}\|d_2\|_{\infty}R^{q^\ast-q}<\frac{1}{2}.
\end{align*}
Let
\begin{align*}
\lambda_{0} :=
\min\left\{\frac{R^{p-1}}{2^{p+1}\|g\|_{\frac{3q}{4q-3}}},\
\frac{R^{q-1}}{2^{q}\|g\|_{\frac{3q}{4q-3}}} \right\}.
\end{align*}
Then for all $0<\lambda\leqslant\lambda_{0}$, we can derive from
\eqref{e3.7} and \eqref{e3.8} that  $\|\mathcal {G} u\| \leqslant
R$. This completes the proof of Lemma \ref{lem3.5}.\end{proof}

\begin{proof}[\bf Proof of Theorem  \ref{T1.1}]
It suffices to  show that $\mathcal {G}
:(W,\preceq)\rightarrow(W,\preceq)$ is an increasing operator, since
by Proposition \ref{pro2.1} and Lemma \ref{lem3.5}, we can
then obtain the existence of the weak solutions.  So let us show
that the operator $\mathcal {T}: (W, \preceq)\rightarrow (W', \lhd)$
is increasing. In fact, take $u_{1}, u_{2}\in W$ such that
$u_{1}\leqslant u_{2}$ almost everywhere on $\mathbb{R}^3$. Due to
the assumptions on $h$ in Theorem \ref{T1.1} and the definition of
operator $\mathcal {T}$, we get
\begin{align}\label{e3.9}
\langle \mathcal {T}u_{1},v\rangle
&=\int_{\mathbb{R}^{3}}\B(\phi_{u_1^+}|u_1^+|^{q-2}u_1^+ 
+h(x,u_1^+)
+\lambda
g(x)\B) v\mathrm{d}x\nonumber\\&\leqslant
\int_{\mathbb{R}^{3}}\B(\phi_{u_2^+}
| u_2^+|^{q-2}u_2^+
+h(x,u_2^+)
+\lambda
g(x)\B) v\mathrm{d}x=\langle \mathcal {T}u_{2},v\rangle,\ \text{for
all}\ v\in \mathcal {E}_{+}.
\end{align}
Therefore, the operator $\mathcal
{G}:(W,\preceq)\rightarrow(W,\preceq)$ is indeed increasing. By
Proposition \ref{pro2.1}  and Lemma \ref{lem3.5}, the operator
$\mathcal {G}$ has a fixed point, that is, there exists $u_{0}\in
B_{W}[0,R]$ such that $\mathcal {G}u_{0}=u_{0}$. Since
$\mathcal {G}=\mathcal {B}^{-1}\circ \mathcal T,$ we have $ \langle
\mathcal {B}u_{0}, v\rangle=\langle \mathcal {T}u_{0} , v\rangle$,
for all $v\in W$. That is,
\begin{align}\label{e3.10}
&\nonumber \int_{\mathbb{R}^{3}}(|\nabla u_0|^{p-2}\nabla u_0\nabla
v + |u_0|^{p-2}u_0v)\mathrm{d}x + \int_{\mathbb{R}^{3}}(|\nabla
u_0|^{q-2}\nabla u_0\nabla v +
|u_0|^{q-2}u_0v)\mathrm{d}x \\
& = \int_{\mathbb{R}^{3}}\phi_{u_0^+}|u_0^+|^{q-2}u_0^+v \mathrm{d}x
-\int_{\mathbb{R}^3}\B( h(x,u_0^+)+\lambda
g(x)\B)v\mathrm{d}x.
\end{align}
Letting $v = u_0^-$ in \eqref{e3.10}, we can get
$\|u_0^-\|_{W^{1,p}(\mathbb{R}^3)}^{p} +
\|u_0^-\|_{W^{1,q}(\mathbb{R}^3)}^{q} = 0$, which means that $u_0^- =
 0$ and is a nontrivial nonnegative weak
solution of problem \eqref{e1.1}. According to the well-known Strong
Maximum Principle, $u_0$ is a positive solution to problem
\eqref{e1.1}.
This completes the proof of Theorem \ref{T1.1}.
\end{proof}

\subsection*{Acknowledgments}
Song was supported by the National Natural Science Foundation of
China (No.12001061), the Research Foundation of Department of
Education of Jilin Province (JJKH20220822KJ) and the Natural Science
Foundation of Jilin Province (Grant No. YDZJ202201ZYTS582,
222614JC010793935). Repov\v{s} was supported by the Slovenian
Research agency grants P1-0292,
J1-4031,
 J1-4001,
 N1-0278,
 N1-0114,
  and N1-0083.
  We gratefully acknowledge the reviewers for their valuable comments and suggestions.

\end{document}